\documentclass[onefignum,onetabnum]{siamart171218}

\usepackage[mathlines,displaymath]{lineno}

\usepackage{lipsum}
\usepackage{amsfonts}
\usepackage{graphicx}
\usepackage{epstopdf}
\usepackage{algorithmic}
\usepackage{bm}
\ifpdf
  \DeclareGraphicsExtensions{.eps,.pdf,.png,.jpg}
\else
  \DeclareGraphicsExtensions{.eps}
\fi

\newsiamremark{remark}{Remark}
\newsiamremark{hypothesis}{Hypothesis}
\crefname{hypothesis}{Hypothesis}{Hypotheses}
\newsiamthm{claim}{Claim}
\newsiamremark{conjecture}{Conjecture}
\headers{Randomized XR-Based CholeskyQR}{Fan, Guo and Lin}

\title{A Novel Randomized XR-Based Preconditioned CholeskyQR Algorithm}

\author{Yuwei Fan\thanks{Theory Lab, Huawei Hong Kong Research Center, Sha Tin, Hong Kong SAR, People’s Republic of China. (\email{fanyuwei2@huawei.com})}\and Yixiao Guo\thanks{LSEC, Institute of Computational Mathematics and Scientific/Engineering Computing, AMSS, Chinese Academy of Sciences. School of Mathematical Sciences, University of Chinese Academy of Sciences, Beijing 100049, People's Republic of China. (\email{guoyixiao@lsec.cc.ac.cn})} \and Ting Lin\thanks{School of Mathematical Sciences, Peking University, Beijing 100871, People's Republic of China. (\email{lintingsms@pku.edu.cn})}}

\usepackage{amsopn}

\DeclareMathOperator{\cond}{cond}

\newcommand{\cholqr}{\texttt{CholeskyQR}}
\newcommand{\chol}{\texttt{Cholesky}}
\newcommand{\alucholqr}{\texttt{rLU-CholeskyQR}}
\newcommand{\aqrcholqr}{\texttt{rQR-CholeskyQR}}
\newcommand{\lu}{\texttt{LU}}
\newcommand{\qr}{\texttt{QR}}

\newcommand{\R}{\mathbb{R}}
\renewcommand{\P}{\mathbb{P}}
\newcommand{\E}{\mathbb{E}}
\renewcommand{\u}{\textbf{u}}

\makeatletter
\newcommand*{\addFileDependency}[1]{% argument=file name and extension
  \typeout{(#1)}% latexmk will find this if $recorder=0 (however, in that case, it will ignore #1 if it is a .aux or .pdf file etc and it exists! if it doesn't exist, it will appear in the list of dependents regardless)
  \@addtofilelist{#1}% if you want it to appear in \listfiles, not really necessary and latexmk doesn't use this
  \IfFileExists{#1}{}{\typeout{No file #1.}}% latexmk will find this message if #1 doesn't exist (yet)
}
\makeatother

\usepackage{multicol,multirow}
\ifpdf
\hypersetup{
  pdftitle={A Novel Randomized XR-Based Preconditioned CholeskyQR Algorithm},
  pdfauthor={Yuwei Fan \and Yixiao Guo\and Ting Lin}
}
\fi

\begin{document}

\maketitle

% REQUIRED
\begin{abstract}
CholeskyQR is a simple and fast QR decomposition via Cholesky decomposition, while it has been considered highly sensitive to the condition number.
In this paper, we provide a randomized preconditioner framework for CholeskyQR algorithm. Under this framework, two methods (randomized LU-CholeskyQR and randomized QR-CholeskyQR) are proposed and discussed. 
We prove the proposed preconditioners can effectively reduce the condition number, which is also demonstrated by numerical tests. 
Abundant numerical tests indicate our methods are more stable and faster than all the existing algorithms and have good scalability.
\end{abstract}

% REQUIRED

% REQUIRED

% \linenumbers

\section{Introduction}
\label{sec:intro}

QR decomposition has laid the foundation of matrix computation, including (weighted) least square problem, block orthogonalization, Krylov subspace methods, randomized singular value decomposition. Based on its essential role in numerical linear algebra, it has shown great power in other areas like MIMO detection, GPS positioning, and so on. The (reduced) QR decomposition reads as 
\begin{equation}
\label{eq:qr}
A = QR,\quad A \in \R^{m\times n}, Q \in \R^{m\times n}, R \in \R^{n\times n}, ~~(m>n)
\end{equation}
where $Q$ is an orthogonal matrix, and $R$ is an upper-triangular matrix. 

Therefore, how to perform QR decomposition quickly and decently has been a great challenge for computer scientists and mathematicians over decades. 
On a modest scale, Householder reflection \cite{householder1958unitary} and Givens rotation are preferred, as they are proven stable, see \cite{golub2013matrix} and \cite[Chapter 18]{higham2002accuracy} for details. 
However, several drawbacks of both methods (mainly Householder QR) have been mentioned in the literature. First, these methods are BLAS1 and BLAS2 intensive, which will be inferior for the machine supporting fast BLAS3 operation. 
Second, such communication-frequent methods might {be poorly performed} for distributed architectures. 
Many studies focus on how to improve Householder QR via BLAS3 and distributed systems, such as \cite{bischof1987wy,bischof1991parallel,granat2010novel,o1990parallel,schreiber1989storage}. 
In contrast, communication-avoiding algorithms \cite{ballard2011minimizing,demmel2013communication} have attracted much attention for their asymptotically optimal communication cost. 
We will not expand this topic, and the interested readers might refer to the references.

In this paper, we focus on the QR decomposition of tall-and-skinny matrices, \textit{i.e.}, the case when $m\gg n$ in \eqref{eq:qr}. 
Such requests could come from randomized singular value decomposition (RSVD) \cite{halko2011finding}, Krylov subspace methods (KSM) \cite{hoemmen2010communication} and local optimal block preconditioned conjugate gradient method (LOBPCG) \cite{duersch2018robust}, as well as block HouseholderQR algorithms \cite{schreiber1989storage}. 
Due to its importance, many efforts have been put in past decades. For direct methods, Tall and Skinny QR (TSQR) was proposed and well studied in the past ten years, see \cite{anderson2011communication,constantine2011tall}. 
However, these improvements are far from satisfactory, mainly due to the complexity of implementation.

On contrast, the CholeskyQR method is famous as an efficient and elegant method for tall-and-skinny matrices, whose history can date back to a century ago. It uses the fact that $R$ factor of $A$ is just the Cholesky factor of $A^TA$. 
Different from Givens and Householder method, the CholeskyQR method requires less communication and computation cost, especially for tall-and-skinny matrices. 
However, the condition number of $A^TA$ is $\cond(A)^2$, which makes the CholeskyQR method suffer from numerical instability, as commented in \cite{fukaya2014choleskyqr2}. 
Hence it is necessary to precondition the CholeskyQR algorithm when $A$ is ill-conditioned. Regarding this, how to choose a suitable preconditioner becomes the main problem of this type of algorithms. 
Here we present a short review of existing work on preconditioning CholeskyQR, and  \cref{sec:algs} expands our discussions. The authors in \cite{fukaya2014choleskyqr2} provide a first and clever way of preconditioning the CholeskyQR algorithm by just performing CholeskyQR again, named CholeskyQR2. In their consequent paper \cite{yamamoto2015roundoff}, rigorous round-off error has been established. However, when $\cond(A) \sim \u^{-1/2}$\footnote{In this paper, $\u$ denotes the machine precision.}, the first Cholesky factorization might break down, leading to an unwanted cessation of CholeskyQR2. To remedy this issue, a shifted CholeskyQR \cite{fukaya2020shifted} is adopted as the preconditioner. More recently, the full LU decomposition is considered as the preconditioner in \cite{terao2020lu}.

Unlike existing deterministic preconditioners, we explore the possibility of randomized ones. It is consensus that modern large-scale computation is eager for randomized algorithms, which have reached tremendous success in linear solvers, range finder, model reduction, \textit{etc.} See \cite{kannan2017randomized,martinsson2020randomized} for more details.

In this paper, we propose a novel framework for preconditioning the CholeskyQR algorithm. The proposed {framework} has the following advantages. First, the method requires less computation and communication costs, comparing to existing preconditioners like the CholeskyQR2 family. Second, the matrix concentration techniques allow us to predict whether the whole system is rank deficient with high probability. If the system is rank deficient, then only direct QR methods could work. Thus in application aspect, our proposed methods are more flexible. Finally, it shares the same merit as other CholeskyQR algorithms: easy to implement. We believe these advantages benefit from randomization techniques, which are powerful and unexplored.

The rest of this paper is organized as follows. In \cref{sec:algs} we recall the definition and historical progress of the CholeskyQR algorithm. Some principles on how to precondition CholeskyQR algorithms are pointed out. Based on these, we propose \emph{randomized XR based CholeskyQR method}, including \emph{randomized LU-CholeskyQR} and \emph{randomized QR-CholeskyQR}. The theoretical analysis is displayed in \cref{sec:proof}. To further corroborate our methods, numerous numerical tests are carried out in 
\cref{sec:experiments}. Two applications are shown in \cref{sec:Applications}, describing how our method can accelerate related algorithms like randomized SVD.

\section{CholeskyQR and its preconditioners}
\label{sec:algs}
This section introduces the family of CholeskyQR algorithms. Hereafter, we assume that $A \in \R^{m\times n}$ is full rank and tall-and-skinny, which means $m \gg n$. We also assume that $l\ge n$ is an integer. 

\subsection{Existing methods review}
The original CholeskyQR is shown in \cref{algs:chol-qr}. It first computes the Grammian matrix of $A$, denoted as $G$, then selects $R$ as the (upper) Cholesky factor. Notice that if no numerical issue happens, the $R$ is exactly the $R$ factor of $A$.\footnote{Since if $A = \tilde Q \tilde R$ is the (reduced) QR decomposition, then the orthogonality of $Q$ yields that $A^TA = \tilde R^T \tilde R$. It follows from the definition of Cholesky decomposition that $R = \tilde R$.} Finally, to recover $Q$, we need to solve a triangular system additionally. 

\begin{algorithm}[htbp]
\caption{$[Q,R]$ = \cholqr($A$)}
\label{algs:chol-qr}
\begin{algorithmic}[1]
\STATE $G = A^TA$.
\STATE $R = \chol(G)$.
\STATE $Q = AR^{-1}$.
\end{algorithmic}
\end{algorithm}

Compared to the classical HouseholderQR and TSQR, the CholeskyQR algorithm has several highlight points. 

\begin{itemize}
    \item First, its computational cost is about half of TSQR and HouseholderQR. 
    \item Second, for the parallel environment, CholeskyQR only needs simple reduction operations, while TSQR and other versions of parallelized QR need a huge amount of reduction operations. 
    \item The CholeskyQR algorithm uses BLAS3 operations, which Housholder QR and TSQR can hardly apply.
\end{itemize}

As a result, CholeskyQR always runs faster than classical QR and TSQR. However, CholeskyQR is prohibitive in real scenarios, and there are several reasons displayed in the literature. 

\begin{itemize}
    \item First, the numerical result seems rather sensitive to the condition number of the matrix. {It was shown in \cite{yamamoto2015roundoff} that the round-off error of CholeskyQR is about $\cond(A)^2\u$.} 
    \item Second, the CholeskyQR method needs to compute the Cholesky factorization of $A^TA$, while the condition number of the latter is $\cond(A)^2$. When $A$ has a large condition number such that $\cond(A^TA)$ is larger than $\u^{-1}$, then the CholeskyQR algorithm will break down at the second step. 
\end{itemize}

CholeskyQR2, proposed in \cite{fukaya2014choleskyqr2}, aimed to improve numerical stability by repeating twice CholeskyQR, see \cref{algs:cholqr2}. In the paper and its sequential work \cite{yamamoto2015roundoff,yamamoto2016roundoff}, CholeskyQR2 has been established in both theory and application. An obvious advantage is that CholeskyQR2 is simple and easy to implement, with acceptable numerical error, see \cite{yamamoto2016roundoff}: Suppose $\hat Z$ is the numerical result of CholeskyQR2, then we have 
\begin{equation}
    \|\hat Z^T \hat Z - I\|_{F} \le 6(m+n+1)n\u ,
\end{equation}
under mild condition.
However, when $\cond(A) = O(\u^{-1/2})$, the CholeskyQR2 also breaks down like the original one. 

\begin{algorithm}[htbp]
\caption{$[Q,R]$ = \cholqr2($A$)}
\label{algs:cholqr2}
\begin{algorithmic}[1]
\STATE $[Q_1,R_1]$ = \cholqr($A$).
\STATE $[Q,R_2]$ = \cholqr($Q_1$).
\STATE $R = R_2R_1$. \footnotemark
\end{algorithmic}
\end{algorithm}

\footnotetext{For orthogonalization, we might only interested in $Q$ factor. Hence there is no need to reconstruct $R$ factor. We refer this to \emph{$R$-less} QR factorization.}

To this end, shifted CholeskyQR3 (\cref{algs:shiftcholqr3}) was proposed to remedy the problem that the first Cholesky step might break down. In \cite{fukaya2020shifted}, the authors proved that the condition number after first Cholesky step can be reduced to 
\begin{equation}
    \cond(\hat{Q}) \le 2\sqrt{1+\frac{\varepsilon}{\|X\|_2^2}(\cond(X)^2)}\sqrt{3},
\end{equation}
under mild condition and suggested that $\varepsilon = 11(mn+n(n+1))\u \|X\|_2^2$ is enough for performance. Often, the {numerical behavior of shifted CholeskyQR} is poor and additional CholeskyQR step needs to be performed twicely. That's why the algorithm is called shifted CholeskyQR3.

\begin{algorithm}[htbp]
\caption{$[Q,R]$ = \texttt{sCholeskyQR3}($A$)}
\label{algs:shiftcholqr3}
\begin{algorithmic}[1]
\STATE $G = A^TA$.
\STATE $R_1 = \chol(G + \epsilon I )$.
\STATE $Q_1 = AR_1^{-1}$
\STATE $[Q_2,R_2]$ = \cholqr($Q_1$).
\STATE $[Q,R_3]$ = \cholqr($Q_2$).
\STATE $R = R_3R_2R_1$
\end{algorithmic}
\end{algorithm}

We comment on these existing methods as a summary, which also inspire our proposed algorithms. Both CholeskyQR2 and shifted CholeskyQR3 can be regarded as  a preconditioner of CholeskyQR, as pointed out in \cite{fukaya2020shifted}. However, there are also some departures from the classical preconditioners. First, the preconditioning techniques are by their nature inexact, like diagonal scaling or incomplete LU decomposition. The might not be true for the CholeskyQR2 algorithm:  the first Cholesky step will recover exact QR decomposition if no numerical round-off error appears. Second, the preconditioning seems costly, due to the numerical instability of the original CholeskyQR: The algorithm needs repeating twice or even more to get a satisfactory result. For these two reasons, those preconditioners are not so familiar with commonly discussed ones.

\subsection{XR-Based preconditioned CholeskyQR}
The purpose of this subsection is to develop new methods for preconditioning the CholeskyQR algorithm. Although CholeskyQR2 and shifted CholeskyQR3 are not so efficient, they still provide an alternative way of viewing the CholeskyQR algorithm. Usually, the preconditioned CholeskyQR algorithm can be regarded as a two-step framework.

\begin{algorithm}[htbp]
\label{algs:framework}
\caption{Abstract Framework of Preconditioned CholeskyQR}
\begin{algorithmic}[1]
    \STATE Precondition Step: A Rough QR decomposition $A = X \tilde{R}$. 
    \STATE CholeskyQR Step: $[Q,\hat{R}] = \cholqr(X)$.
    \STATE Recover $R = \hat{R}\tilde{R}$.
\end{algorithmic}
\end{algorithm}

For the CholeskyQR2 algorithm, the preconditioning step is just CholeskyQR itself, which is accurate but less efficient, as mentioned in the previous subsection. The core of such an efficient and stable algorithm is finding a suitable and practical rough QR decomposition. It should meet the following requirements. First, a fast rough QR decomposition is required in this framework. Second, the $R$ factor of rough QR decomposition must be close to the exact $R$ factor and must be upper-triangular. Equivalently, the $Q$ factor of rough QR decomposition must enjoy a mild condition number. Notice that $Q$-less QR decomposition can be used in the preconditioning step since only the $R$ factor is required.

It is necessary to achieve good performance to make $\cond(\tilde X)$ smaller. Ideally, the best rough QR decomposition must be the exact QR decomposition, like what we encounter in the CholeskyQR2 family, regardless of round-off error. 
However, such a step is unaffordable since it achieves our ultimate goal, and no stable QR decomposition can be applied for large-scale matrices.

Therefore, we attempt to perform XR decomposition, in a submatrix $A_1 \in \R^{l\times n}$. The submatrix extracts the core information of the original matrix. Denote by $P$ the projection operator, such that $A_1 = PA$. Then we assume that $\tilde R$ is the $R$ factor of $PA$, and $X:= A\tilde R^{-1}$. Therefore, XR decomposition fits into the framework in \cref{algs:framework}. In this paper,
XR might be LU or QR, to make sure that $R$ is preserved upper-triangular. Plus, $l\ge n$ is assumed to keep the correct size of $R$. For this purpose, we propose the following two algorithms, attempting to use randomized strategy to reduce the condition number of the CholeksyQR algorithm. 

More precisely, our proposed method includes two steps: We first choose a submatrix $A_1$, then perform {LU or QR decomposition to get $R$, the upper-triangular factor of $A_1$}. If $R$ is invertible, then $R^{-1}$ serves as the preconditioner: we calculate $X = AR^{-1}$, and the final $Q = \cholqr(X)$ is our desired $Q$ factor. These two algorithms are called randomized LU-CholeskyQR and randomized QR-CholeskyQR algorithms, respectively. Detailed algorithms are displayed in \cref{alg:rlu-CholeskyQR,alg:rqr-CholeskyQR}.

\begin{algorithm}[htbp]
\caption{$[Q,R]$ = \alucholqr($A$)}
\label{alg:rlu-CholeskyQR}
\begin{algorithmic}[1]
\STATE Choose $A_1$ as a submatrix of $A$, with size $l\times n$.
\STATE $[L,U] = \lu(A_1)$.
\STATE $X = AU^{-1}$.
\STATE $[Q,R_2]$ = \cholqr($X$).
\STATE $R = R_2U$.
\end{algorithmic}
\end{algorithm}

\begin{algorithm}[htbp]
\caption{$[Q,R]$ = \aqrcholqr($A$)}
\label{alg:rqr-CholeskyQR}
\begin{algorithmic}[1]
\STATE Choose $A_1$ as a submatrix of $A$, with size $l\times n$.
\STATE $[\sim,R_1] = \qr(A_1)$.\footnotemark
\STATE $X = AR_1^{-1}$. 
\STATE $[Q,R_2]$ = \cholqr($X$).
\STATE $R = R_2R_1$.
\end{algorithmic}
\end{algorithm}

\footnotetext{$[\sim,R]$ represents a Q-less QR decomposition algorithm.}

We make some remarks on the choice of XR decomposition. We always prefer QR decomposition, in the sense that when $l = n$, the rough QR decomposition becomes exact. However, even the exact LU decomposition cannot ensure that either $L$ or $U$ factor is better-conditioned than $LU$ itself, and such attempt is considered in \cite{terao2020lu}.

The proposed framework has several advantages. First, the partial XR factorization is cheap and easy to implement, and the computational cost is always $O(ln^2)$. Later we will show that it is sufficient to choose $l$ slightly larger than $n$, yielding the computational cost of rough QR decomposition approximately $O(n^3)$. Second, in general, the partial $R$ factor is a suitable preconditioner, and the corresponding $Q$ factor behaves well-conditioned, similarly to CholeskyQR2 and the shifted CholeskyQR3.

Compared to the CholeskyQR2 family, our proposed preconditioners are more similar to the classical ones in the following two senses:
\begin{itemize}
    \item The preconditioning step is inexact, even if no round-off error is involved.
    \item The preconditioning step is much cheaper than the CholeskyQR step.
\end{itemize}

It is essential in our methods that sacrificing the accuracy in the preconditioning step gains much flexibility. That's why we prefer randomized algorithms to deterministic ones.

However, it is still a CholeskyQR method, and some common drawbacks seem inevitable currently. The success of our algorithm requires the full rank assumption of the original matrix. Moreover, the algorithm might break down when $R_1$ is degenerate even if $A$ is of full rank due to bad sampling. Luckily, the most scenario we have considered will not reach the worst case, for example, randomized SVD, see \cref{sec:experiments,sec:Applications}.

\subsection{Strategies of choosing submatrix}
\label{sec:algs:submatrix}

This subsection deals with how to choose the submatrix $A_1$. We describe several strategies, which are all prevailing and well-studied in the past years. This choice is also called random projection in the literature, see \cite{gittens2013topics} and references therein for further study. We here summarize two common techniques.
We introduce the unified expression $A_1 = PA$, where $P \in \R^{l \times m}$ will be explained case by case.

\paragraph{Row Extraction}
Row extraction generates the submatrices from randomly choosing $l$ rows from original $m$ rows directly. Either equi-probability or a priori weighted sampling can be used. In this case, $P$ is a row-chosen matrix. The advantage of row extraction is easy to implement and requires less communication time and cost in parallel. The disadvantage is that it will provide poor preconditioning in some extreme cases, as pointed out in the later analysis. However, these extreme cases encounter rarely in most practical cases. Hence we adopt this strategy in \cref{sec:experiments} for simplicity.

\paragraph{Gaussian Ensemble}
Gaussian Ensemble generates the submatrices from $A_1 = \Omega A$, where $\Omega \in \R^{l \times m}$ is a Gaussian matrices (that means each entry of $\Omega$ is drawn from $N(0,1)$ i.i.d.) This method enjoys the optimal theoretical result, but the multiplication operation is unaffordable. {Hence we only treat it as a theoretical supplement for our proposed methods.}

\subsection{Parallelism}
This subsection discusses how to perform the proposed algorithms in parallel, typically in an MPI environment. For convenience, we use terminology in collective communication from MPI functions directly, such as \texttt{scatter}, \texttt{broadcast}, \texttt{reduce}, \texttt{gather}. Since the Cholesky decomposition is always assumed to be performed in a single processor, it suffices to discuss computing $A^TA$, \textit{i.e.} forming the Grammian matrix, and compute $AR^{-1}$ in parallel environment.

\begin{itemize}
\item Forming global Grammian matrices is divided into two steps: 1) Each processor computes its local Grammian matrix $A_{loc}^TA_{loc}$. 2) Use \texttt{reduce} to summing all local Grammian matrix up, in root processor. 

\item Solving the upper-triangular systems also includes two steps: 1) The root processor sends $R$ to each processor using \texttt{broadcast}. 2) Each processor solves the upper-triangular systems individually.

\end{itemize}

\subsection{Complexity analysis}
\label{sec:overhead}
In this subsection, we study the computation and communication costs of our algorithms. We compare our algorithms with other existing methods, like HouseholderQR, CholeskyQR2 family, and TSQR. We suppose that $A \in \mathbb{R}^{m \times n}$, $A_1 \in \mathbb{R}^{l \times n}$, and $p$ be the number of processes. In complexity analysis, the low order term $O(m^2+n^2+l^2)$ will be omitted.

\subsubsection{Computational costs}
The analysis consists of two parts, the analysis of CholeskyQR and randomized QR.

\paragraph{CholeskyQR} We first evaluate the complexity of computing Grammian matrix in parallel. Each processor first computes the local Grammian matrix, which needs $n^2m/p$ operations per processor. Then we perform Cholesky decomposition in root node, which needs $\frac{1}{3}n^3 + l.o.t.$ ops (only root node). Finally, solving a triangular system needs $C_{\texttt{tri}} = n^2m/p$ ops per processor. Hence $C_{\cholqr} = 2n^2m+\frac{1}{3}n^3$.

\paragraph{Precondition Step \cite{golub2013matrix}} For LU factorization, the cost is $C_{\texttt{LU}} = n^2(l-\frac{1}{3}n)$, and for QR factorization, we have $C_{\texttt{QR}} = n^2(4l - \frac{4}{3}n)$.

\paragraph{Summary} The analysis result is summarized in the table below. We should keep in mind that $m \gg l > n$, and terms involving $m$ will be dominant. Notice that this table only discusses the cost about computing $Q$ factor, and we do not discuss forming $R$ factor by matrix multiplication.

\begin{table}[htbp]
	\centering
	\begin{tabular}{|c|c|c|}
	\hline
					& Total                         & Along Critical Path           \\ \hline
	HouseholderQR   & $n^2(4m - \frac{4}{3}n)$        &    --                         \\ \hline
	CholeskyQR      & $n^2(2m+\frac{1}{3}n)$        & $n^2(2m/p+\frac{1}{3}n)$      \\ \hline
	CholeskyQR2     & $n^2(4m+\frac{2}{3}n)$        & $n^2(4m/p+\frac{2}{3}n)$      \\ \hline
	sCholeskyQR3    & $n^2(6m+n)$                   & $n^2(6m/p+n)$                 \\ \hline
	% rLU-CholeskyQR  & $3n^2m + 4n^2l -n^3$          & $3n^2m/p + 4n^2l - n^3$       \\ \hline
	rLU-CholeskyQR  & $n^2(3m + l)$          & $n^2(3m/p + l) $       \\ \hline
	rQR-CholeskyQR  & $n^2(3m + 4l -n)$          & $n^2(3m/p + 4l -n)$       \\ \hline
	\end{tabular}

% \begin{table}[htbp]
% \centering
% \begin{tabular}{|c|c|c|}
% \hline
%                 & Total                         & Along Critical Path           \\ \hline
% HouseholderQR   & $4n^2m-\frac{4}{3}n^3$        &    --                         \\ \hline
% CholeskyQR      & $2n^2m+\frac{1}{3}n^3$        & $2n^2m/p+\frac{1}{3}n^3$      \\ \hline
% CholeskyQR2     & $4n^2m+\frac{2}{3}n^3$        & $4n^2m/p+\frac{2}{3}n^3$      \\ \hline
% sCholeskyQR3    & $6n^2m+n^3$                   & $6n^2m/p+n^3$                 \\ \hline
% rLU-CholeskyQR  & $3n^2m + 4n^2l -n^3$          & $3n^2m/p + 4n^2l - n^3$       \\ \hline
% \gyx{rLU-CholeskyQR}  & $3n^2m + n^2l$          & $3n^2m/p + n^2l $       \\ \hline
% rQR-CholeskyQR  & $3n^2m + 4n^2l -n^3$          & $3n^2m/p + 4n^2l - n^3$       \\ \hline
% \end{tabular}
\caption{Computational Cost Analysis (l.o.t omitted)}
\end{table}

%\paragraph{Comparing to CholeskyQR2} When $4l < m + \frac{1}{3}n$, the total number arithmetic operations of \texttt{aqr-chol-qr} is smaller than \texttt{chol-qr2}. The same result holds true when along the critical path, provided $4l < m/p + \frac{1}{3}n$.

\subsubsection{Communication costs}
In this subsection, we compare the communication costs simply by two indices: the communication times and the total size of communicated data. We omit the comparison with TSQR since we only consider collaborative communication.

\begin{table}[htbp]

\centering
\begin{tabular}{|c|c|c|}
\hline
               & \# Times                           & \# Data               \\ \hline

CholeskyQR        &  2     & $pn^2   $      \\ \hline
CholeskyQR2       & 4          &  $2pn^2$        \\ \hline
sCholeskyQR3      & 6           & $3pn^2$                 \\  \hline
rLU-CholeskyQR    &    $3\sim 4$           &   $ \frac{3}{2}pn^2\sim \frac{3}{2}pn^2 + nl$                          \\ \hline
rQR-CholeskyQR    & $3\sim 4$ & $\frac{3}{2}pn^2 \sim \frac{3}{2}pn^2 + nl$ \\ \hline
\end{tabular}
\caption{Communicatoin Cost Analysis (l.o.t omitted)}
\end{table}

\section{Theoretical analysis}
\label{sec:proof}
In this section, we provide some theoretical analysis on {randomized QR-CholeskyQR algorithm}. The main tool used in this section is concentration inequalities and related topics, and the readers not familiar with those might refer to \cite{vershynin2018high,wainwright2019high} for a comprehensive study. For ease of understanding, we omit the round-off error in this section. Essentially, there seems no need to estimate the round-off error provided both small QR and triangular solver are stable. However, for completeness, we give an analysis in \cref{app:round-off-error}.

 We first simplify the problem via a matrix analysis approach. We first recall some basic definitions and properties of orthogonal matrices and the condition number. Here, we denote by $\|\cdot\|$ the  $\ell^2$ norm of a vector. Throughout this section, $Q \in \mathbb{R}^{m\times n}$ is called an orthogonal matrix if $Q^TQ = I_n$, where $I_n \in \mathbb{R}^{n \times n}$ is the identity matrix. Notice that this implies $m \ge n$. The condition number is under $\ell^2$ norm, that is

\begin{equation} \cond(M) = \frac{\sigma_{\max}(M)}{\sigma_{\min}(M)} ,\end{equation}
where $\sigma_{\max}$ and $\sigma_{\min}$ are the largest and smallest singular value of matrix $A$. A consequent result is that if $Q$ is an orthogonal matrix, then $\cond(QM) = \cond(M)$, since multiplying an orthogonal matrix does not change the singular value. For QR decomposition $A = QR$, we have $\cond(A) = \cond(R)$.

Suppose $A {\in \mathbb{R}^{m\times n}}$ is full rank, with QR decomposition $A = QR$, where $Q \in \mathbb{R}^{m\times n}$ is orthogonal and $R \in \mathbb{R}^{n \times n}$ is upper-triangular. Suppose $P$ is the row-chosen matrix such that the submatrix $A_1$ can be expressed as $A_1 = PA$, see \cref{sec:algs:submatrix}. Our main result starts from the following proposition, 

\begin{proposition}
    \label{prop:cond-number-identity}
Suppose that $R_1$ is the $R$ factor of $A_1 = PA$, and $X:=AR_1^{-1}$. Moreover, suppose that $A = QR$ is its QR decomposition, then we have
 \begin{equation}\cond(X) = \cond(PQ),\end{equation}
 provided $A$ and $A_1$ are full rank.
\end{proposition}
\begin{proof}

Since $A_1 = PA = PQR$, $A_1$ is full rank is equivalent to $PQ$ is full rank.  Define  $\hat Q_1: = A_1R_1^{-1}$ is $Q$-factor of $A_1$, by definition it is an orthogonal matrix. Suppose $PQ = Q_2R_2$ is the QR decomposition of $PQ$.
Since $A = QR$, we have 
\begin{equation}\hat{Q}_1R_1 = A_1 = PA = PQR = Q_2R_2R.\end{equation} Observe that by assumption we have $\hat{Q}_1, R_1, Q_2, R$ are all invertible. It follows from the uniqueness of QR decomposition that $R_1 = R_2R$. Hence \begin{equation}X = AR_1^{-1} = QRR_1^{-1} = QR_2^{-1}.\end{equation} Since $Q$ is orthogonal matrix, we have
\begin{equation}\cond(X) = \cond(R_2^{-1}) = \cond(R_2) = \cond(PQ).\end{equation}
\end{proof}

% Now we are ready to estimate the condition number of $\cond(X) = \cond(PQ)$. 
Let $\lambda_{\max}(X)$ and $\lambda_{\min}(X)$ be the largest and smallest eigenvalue of a matrix $X$, notice that they coincide with $\sigma_{\max}(X)$ and $\sigma_{\min}(X)$ when $X$ is symmetric and semi-positive definite. Recall the matrix Chernoff inequality from \cite[Theorem 5.1.1]{tropp2015introduction}:
\begin{proposition}[Matrix Chernoff Inequality]
\label{prop:chernoff}
Suppose $X_1,\cdots, X_s  \in \mathbb{R}^{n\times n}$ are $s$ independent, symmetric, random matrices with dimension $n$.\footnote{The original statement is for Hermitian matrices.} Suppose that 
$0 \le \lambda_{\min}(X_k) \le \lambda_{\max}(X_k) \le L$ holds for all $k$. 

Define $Y = \sum_{k=1}^s X_k$ be their sum, and denote by $\mu_{\min} := \lambda_{\min}(\E Y), \mu_{\max} := \lambda_{\max}(\E Y)$. Then for any $\varepsilon >0$ we have
\begin{equation}
    \P\bigg[\lambda_{\min}(Y) \le (1-\varepsilon)\mu_{\min} \bigg] \le n\left[\frac{e^{-\varepsilon}}{(1-\varepsilon)^{1-\varepsilon}}\right]^{\mu_{\min}/L} 
    \end{equation}
and 
\begin{equation}
    \P\bigg[\lambda_{\max}(Y) \ge (1+\varepsilon)\mu_{\max}\bigg] \le n\left[\frac{e^{\varepsilon}}{(1+\varepsilon)^{1+\varepsilon}}\right]^{\mu_{\max}/L}.
    \end{equation}
\end{proposition}

Based on this proposition, we show our first result, giving an estimate on the condition number of $X$ when each row is chosen with equal probability. Before we go to the details of the statement and proof of this theoretical result, let us consider some extreme cases: suppose $A = \begin{bmatrix}I_{n}\\ \mathbf{0}_{(m-n)\times n}\end{bmatrix}$, then the submatrix will be considerably ill-conditioned, since indeed we need to choose each of the first $n$ rows of $A$ at least one time, in our row choosing strategy. This only occurs in a probability about 
\[\binom{m - n}{l - n}/\binom{m}{l} = \frac{(l)(l-1)\cdots(l-n)}{(m)(m-1)\cdots(m-n)} \approx (\frac{l}{m})^n \ll 1.\]

Therefore, the behavior of our method is strongly related to the structure of the original matrix $A$. To characterize it more precisely, we introduce $\theta(Q) = \max_{i}(\|Q(i,:)\|_2^2)$\footnote{Here we adopt the MATLAB notation. For example, $Q(i,:)$ denotes the $i$-th row of $Q$.}, aiming to measure how the elements of a matrix concentrate.

\begin{theorem}[Equal Probability Case]
    \label{thm:equal-prob-case}
Suppose $A \in \mathbb{R}^{m \times n}$ with QR decomposition $A = QR$, the submatrix is generated by choosing $l$ row independently, while each row of submatrix is drawn from the original matrix $A$ with equal probability $ = \frac{1}{m}$. 
Then for any $\delta > 0, \varepsilon \in (0,1)$, we have 
\begin{equation}
    \P\left[ \cond(X) \ge \sqrt{\frac{1+\delta}{1-\varepsilon}}\right] \le n\left[\left[\frac{e^{-\varepsilon}}{(1-\varepsilon)^{1-\varepsilon}}\right]^{l/(m\theta(Q))} +\left[\frac{e^{\delta}}{(1+\delta)^{1+\delta}}\right]^{l/(m\theta(Q))} \right],
\end{equation}
where $X$ is defined in \cref{alg:rqr-CholeskyQR}.
\end{theorem}

\begin{proof}
We first construct a matrix-valued random variable $X$, to adapt the result in \cref{prop:chernoff}. Separate $Q$ into $m$ row vectors $Q = [\bm q_1^T, \bm q_2^T, \cdots, \bm q_m^T]^T$, where $\bm q_i := Q(i,:)$ is the $i-$th row vector. Define 
\[ M = \bm q_k^T\bm q_k \in \mathbb{R}^{n \times n} \text{ with probability } \frac{1}{m} \text{ for each } k = 1,2,\cdots,m.\]
Then we have 
\begin{equation}
\E M = \frac{1}{m}(\bm q_1^T\bm q_1 +\cdots + \bm q_m^T\bm q_m) = \frac{1}{m}Q^TQ = \frac{1}{m}I_n,
\end{equation}
and it follows from $\lambda_{\max}(\bm q^T \bm q)\le \|\bm q\|^2$ that the following inequality holds,
\begin{equation}
\lambda_{\max}(M) \le \max_{k} \|\bm q_k\|^2 = \theta(Q),
\end{equation}
from the definition of $\theta(Q)$.

Set $M_1,\cdots, M_l \sim M$ independently, and $Y = M_1+M_2+\cdots M_l$. $\E Y = l~\E M =\frac{l}{m}I_n$, therefore $\lambda_{\min}(\E Y) = \lambda_{\max}(\E Y) = \frac{l}{m}$. Here we point out that $\cond(X) = \sqrt{\cond(Y)}$, which will be proved later. Hence it suffices to estimate $\cond(Y)$. By applying \cref{prop:chernoff} we obtain 

\begin{equation}
    \P\bigg[\lambda_{\min}(Y) \le (1-\varepsilon)\frac{l}{m} \bigg] \le n\left[\frac{e^{-\varepsilon}}{(1-\varepsilon)^{1-\varepsilon}}\right]^{\frac{l}{m}\frac{1}{\theta(Q)}} 
    \end{equation}
and 
\begin{equation}
    \P\bigg[\lambda_{\max}(Y) \ge (1+\varepsilon)\frac{l}{m}\bigg] \le n\left[\frac{e^{\varepsilon}}{(1+\varepsilon)^{1+\varepsilon}}\right]^{\frac{l}{m}\frac{1}{\theta(Q)}}.
    \end{equation}

Combining both inequalities, we conclude that
\[
    \P\left[ \cond(Y) \ge {\frac{1+\delta}{1-\varepsilon}} \right] \le n\left[\left[\frac{e^{-\varepsilon}}{(1-\varepsilon)^{1-\varepsilon}}\right]^{l/(m\theta(Q))} +\left[\frac{e^{\delta}}{(1+\delta)^{1+\delta}}\right]^{l/(m\theta(Q))} \right].
\]

Finally, we show $\cond(X) = \sqrt{\cond(Y)}$. It follows from \cref{prop:cond-number-identity} that if the submatrix $A_1 = [\bm a_{i_1}^T,\bm a_{i_2}^T,\cdots, \bm a_{i_l}^T]^T$, where $\bm a_k$ is the $k-$th row vector of $A$, then $\cond(X) = \cond(Q_1)$, where $Q_1 = [\bm q_{i_1}^T,\bm q_{i_2}^T,\cdots, \bm q_{i_l}^T]^T.$ Therefore \[Q_1^TQ_1 = \sum_{s = 1}^{l} = \bm q_{i_s}^T \bm q_{i_s} = Y,\] hence $\cond(Y) = \cond(X)^2$.
\end{proof}

 An observation is that the concentration result works if and only if $m\theta(Q)$ is small. It is believable that in general we have $m\theta(Q) = O(n)$, hence if $l = O(n \log n \log \delta^{-1})$ can ensure probability $1 - \delta$ to obtain such a preconditioner. However, this might not hold true for some case. Rethinking that case we have mentioned before, suppose $Q = [I, \,\, 0]^T$ where $I$ is identity matrix. In this case, $\theta(Q) = 1$, and the concentration works if and only if $l = O(m\log n\log \delta^{-1})$, which will eventually impractical when $n$ get larger.

Now consider the case when each row is chosen with unequal probability. The following theorem shows under some strategy, the $m\theta(Q)$ can be sharpened to $n$, which is the optimal rate. Though impractical, this theoretical result might shed light on the mechanism of the row choosing strategy, which seems promising in the future.

\begin{theorem}[Unequal Probability Case]
	\label{thm:unequal-prob-case}
    Suppose $A \in \mathbb{R}^{m \times n}$ with QR decomposition $A = QR$. The submatrix is generated by choosing $l$ row independently, while each row of submatrix is drawn as follows: for each $\bm a_i$, we select $\|\bm q_i\|^{-1} \bm a_i$ with probability $\frac{1}{n}\|\bm q_i\|^2$.
    Then for any $\delta > 0, \varepsilon \in (0,1)$, we have 
\begin{equation}
    \P\{ \cond(X) \ge \sqrt{\frac{1+\delta}{1-\varepsilon}} \} \le n\left[\left[\frac{e^{-\varepsilon}}{(1-\varepsilon)^{1-\varepsilon}}\right]^{l/n} +\left[\frac{e^{\delta}}{(1+\delta)^{1+\delta}}\right]^{l/n} \right],
\end{equation}
where $X$ is defined in \cref{alg:rqr-CholeskyQR}.
\end{theorem}

\begin{proof}
We modify the construction of $M$ in \cref{thm:equal-prob-case}. Notice that \[\sum_{k = 1}^m \|\bm q_k\|^2 = \|Q\|^2 = n.\] Define 
\[M = \frac{1}{\|\bm q_k\|^2} \bm q_k^T\bm q_k \in \mathbb{R}^{n \times n} \text{ with probability } \frac{1}{n}\|\bm q_k\|^2 \text{ for each } k = 1,2,\cdots,m.\]
Then we have 
\begin{equation}
\E M = \frac{1}{n}(\bm q_1^T\bm q_1 +\cdots + \bm q_m^T\bm q_m) = \frac{1}{n}Q^TQ = I_n,
\end{equation}
and the following inequality holds,
\begin{equation}
L := \lambda_{\max}(M) = \max_{k} \|\bm q_k\|^2 = 1,
\end{equation}
from the definition of $\theta(Q)$. The rest of the proof is similar to that in \cref{thm:equal-prob-case}.
\end{proof}

Next, we show the result of the Gaussian ensemble. In actual computation, one might not hope to use Gaussian ensemble for its expansive arithmetic cost. However, what we believe is that this result reveals the numerical behavior of our proposed methods appropriately.

\begin{theorem}[Gaussian Ensemble Case]
Given $A \in \R^{m\times n}$ with QR decomposition $A = QR$, the submatrix is generated by left-multiplying Gaussian matrices $\Omega \in \R^{l\times m}$, where each component is drawn from standard normal distribution $\mathcal{N}(0,1)$ iid. Suppose $l >n$, then we have 
\begin{equation}
    \P \bigg[ \cond(X) \ge \frac{3 +\sqrt{\frac{n}{l}} }{1 - \sqrt{\frac{n}{l}}}\bigg] \le 2 e^{-\frac{(\sqrt{l} - \sqrt{n})^2}{8}},
\end{equation}
where $X$ is defined in \cref{alg:rqr-CholeskyQR}.
\end{theorem}

\begin{proof}
Since $Q^TQ = I_m$, it follows from \cite[Theorem 6.1]{wainwright2019high} that
\begin{equation}
    \P \bigg[\sigma_{\max}(\Omega Q)/ \sqrt{l} > (1+\varepsilon) + \sqrt{\frac{n}{l}} \bigg] \le e^{-\frac{l\varepsilon^2}{2}}
\end{equation}
and 
\begin{equation}
     \P\bigg[\sigma_{\min}(\Omega Q)/\sqrt{l} < (1-\varepsilon) - \sqrt{\frac{n}{l}} \bigg] \le e^{-\frac{l\varepsilon^2}{2}}
\end{equation}

Choose $\varepsilon = \frac{1}{2}(1 - \sqrt{\frac{n}{l}})$, yielding

\begin{equation}
    \P\bigg[\sigma_{\max}(\Omega Q)/ \sqrt{l} >  \frac{3}{2} + \frac{1}{2}\sqrt{\frac{n}{l}} \bigg] \le e^{-\frac{(\sqrt{l} - \sqrt{n})^2}{8}}
\end{equation}
and 
\begin{equation}
     \P\bigg[\sigma_{\min}(\Omega Q)/\sqrt{l} < \frac{1}{2}( 1- \sqrt{\frac{n}{l}}) \bigg] \le  e^{-\frac{(\sqrt{l} - \sqrt{n})^2}{8}}
\end{equation}

It follows from \cref{prop:cond-number-identity} that $\cond(X) = \cond(\Omega Q)$, we conclude that 
\[\P \bigg[ \cond(X) \ge \frac{3 +\sqrt{\frac{n}{l}} }{1 - \sqrt{\frac{n}{l}}}\bigg] \le 2 e^{-\frac{(\sqrt{l} - \sqrt{n})^2}{8}}.\]
\end{proof}
\begin{remark}
If $l > 1.1n$, then simple estimation shows that 
\[\P \bigg[ \cond(X) \ge 100 \bigg] \le 2 e^{-\frac{(\sqrt{l} - \sqrt{n})^2}{8}}.\]
\end{remark}

\paragraph{Summarization} In summary, we provide the theoretical result, including both submatrix choosing strategies and Gaussian ensemble. These results shows that exponential decay might happen when the size of the submatrix becomes larger and larger. Hence, for a variety of matrices, our framework might produce a practical preconditioner. Unfortunately, the exponential decay of row choosing strategy does not hold for some matrices. We believe that this is due to large $\theta(Q)$, \textit{i.e.,} the matrix entries concentrate in a lower-dimensional matrix. 

\section{Numerical results}
\label{sec:experiments}
In this section, we present plenty of numerical experiments to illustrate our results. We compare the performance among HouseholderQR, CholeskyQR, CholeskyQR2 (from \cite{fukaya2014choleskyqr2}), shifted CholeskyQR3 (from \cite{fukaya2020shifted}), and two newly proposed algorithms: randomized LU-CholeskyQR and randomized QR-CholeskyQR. \Cref{sec:experiments:setup} introduces the setup of our experiments, including computing configurations and the strategy of generating test matrices. The next several sections discuss the stability, runtime, strong and weak scalability in \cref{sec:experiments:Accuracy,sec:experiments:speed,sec:experiments:scalability} respectively. In addition, \cref{sec:experiments:sampling} concerned about the relationship between the condition number and the number of sampling rows.

\subsection{Setup}
\label{sec:experiments:setup}

The experiments carried out in this section are all implemented in C++ language, with the help of OpenMPI\footnote{\href{https://www.open-mpi.org/}{https://www.open-mpi.org/}.} version 3.1.4, and BLAS and LAPACK libraries implemented in Intel MKL version 2017.1 on High-performance Computing Platform of Peking University, see \cref{specifications} for specifications. We use the double-precision, hence $\u := 2^{-52} \approx 2.22 \times 10^{-16}$, and the choice of shift in shifted CholeskyQR3 is $10^{-15}$.\footnote{This choice is different as that in \cite{fukaya2020shifted}, since the original choice will lead to a break-down when the matrix has a condition number $\sim 10^{15}$.}

\begin{table}[htbp]
	\centering
	\caption{Specifications of computing configuration.}
	\begin{tabular}{c|c}	
		\hline
		Item & Specification \\ 
		\hline
		CPU & Intel Xeon E5-2697A V4 (2.60GHz, 16 cores) \\
		Number of CPUs / node & 2 \\
		Number of threads / node & 32 \\
		Peak FLOPS / node & 1.33 TFlops \\
		Memory size / node & 256G \\ 
		\hline
	\end{tabular}
	\label{specifications}
\end{table}

We generate test matrices $A$ in the following manner, forming
%\begin{equation}
$	A = U \Sigma V^T \in \mathbb{R}^{m\times n}$,
%\end{equation}
where both $U \in \mathbb{R}^{m \times m}$ and $V \in \mathbb{R}^{n \times n}$ are orthogonal matrices from taking a full SVD of a random matrix, and
% \begin{equation}
$	\Sigma := \mathrm{diag}(1, \sigma^{\frac{1}{n-1}}, \dots, \sigma^{\frac{n-2}{n-1}}, \sigma).$
% \end{equation}
Here $\sigma \in (0,1)$ is a constant. Clearly we have $\|A\|_2 = 1$ and $\cond(A) = \kappa := \frac{1}{\sigma}$.
We denote the number of processes involved by $p$.

\subsection{Accuracy test}
\label{sec:experiments:Accuracy}
First, we examine the numerical stability of randomized LU-CholeskyQR and randomized QR-CholeskyQR (shortened as  \texttt{rLU-CholeskyQR} and \texttt{rQR-CholeskyQR} respectively), when fixing $l = 2n$. Throughout this section, two measurements will be investigated: orthogonality $\|Q^TQ-I\|_F$ and residual $\|A - QR\|_F / \|A\|_F$, where $F$ denotes Frobenius norm.

\Cref{accuracy1,accuracy2} display the numerical stability, where we take $m = 10^5$, $n = 100$, and the condition number $\kappa$ varies from $10^3$ to $10^{15}$. We see in \cref{accuracy1}, the orthogonality error of all methods except CholeskyQR are acceptable, in the sense that those are almost independent to the condition number. The figure indicates that the orthogonality error of CholeskyQR is about $O(\kappa^2 \u)$, only valid when $\kappa \leq 10^8$. The breakdown of CholeskyQR and CholeskyQR2 when $\kappa > 10^8$ has been already observed \cite{fukaya2014choleskyqr2}. For the residual aspect, see \cref{accuracy2}, all methods listed here are stable, expect a slight difference is detected in our tests (HouseholderQR performs a little worse while CholeskyQR performs a little better).

\begin{figure}[htbp]
\begin{center}
	\includegraphics[width=0.8\columnwidth]{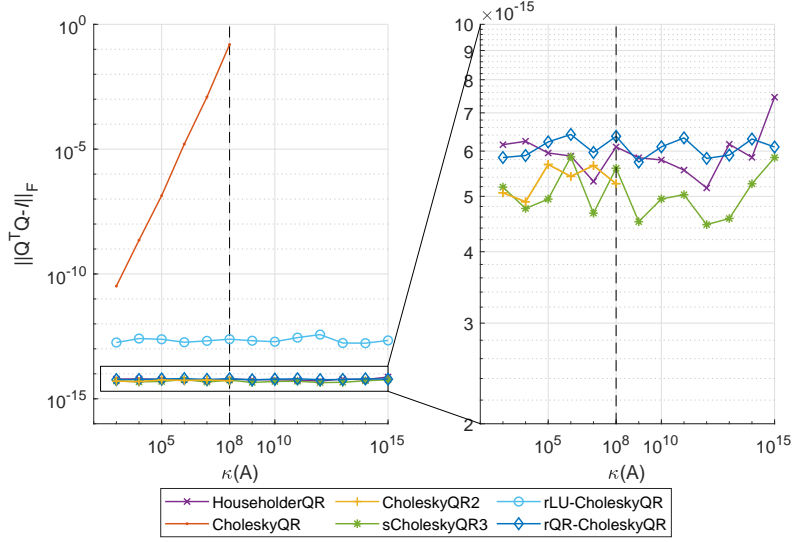}
	\caption{$\| Q^TQ - I\|_F$ for different algorithms when $m=10^5$, $n = 100$. }
	\label{accuracy1}
\end{center}
\end{figure}

\begin{figure}[htbp]
\begin{center}
	\includegraphics[width=0.8\columnwidth]{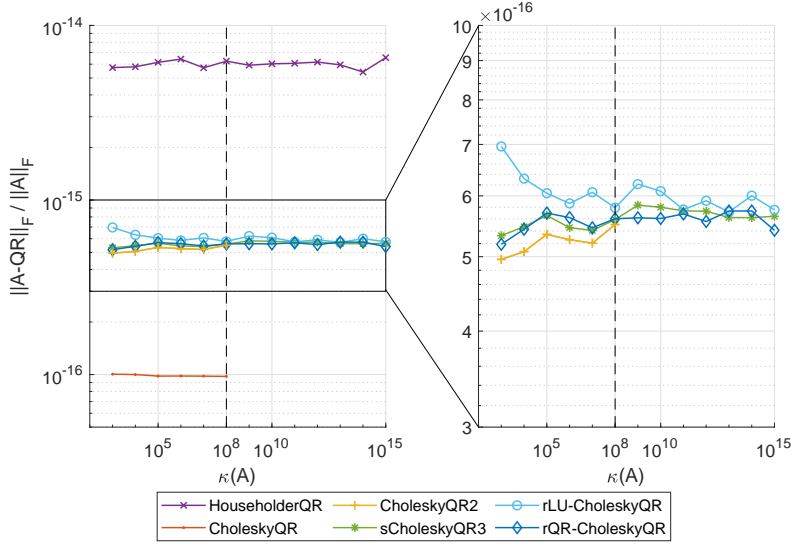}
	\caption{$\frac{\|A - QR\|_F}{\|A\|_F}$ for different algorithms when $m=10^5$, $n = 100$. }
	\label{accuracy2}
\end{center}
\end{figure}

Next, we show the dependence of numerical stability on the size of matrices. We take $\kappa = 10^{5}$ and fix $m$ and $n$ respectively, see \cref{accuracy3,accuracy4,accuracy5,accuracy6}. It follows from the numerical results that CholeskyQR is unstable due to the large condition number. The rLU-CholeskyQR algorithm performs a little worse, while the orthogonality error of other methods increases with $m$ and $n$ only mildly, and all stable. All tests support that the residual is negligible for most practical cases.

\begin{figure}[htbp]
\begin{center}
	\includegraphics[width=0.8\columnwidth]{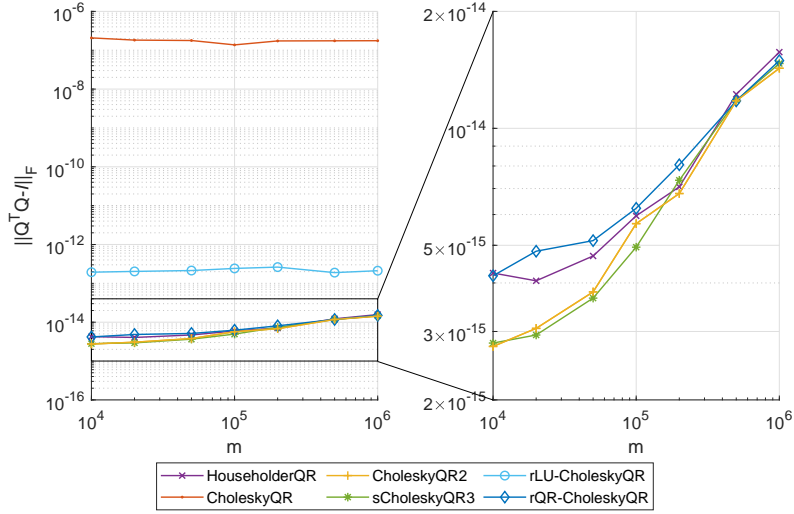}
	\caption{$\| Q^TQ - I\|_F$ for different algorithms when $n = 100$, $\kappa = 10^{5}$.}
	\label{accuracy3}
\end{center}
\end{figure}

\begin{figure}[htbp]
\begin{center}
	\includegraphics[width=0.8\columnwidth]{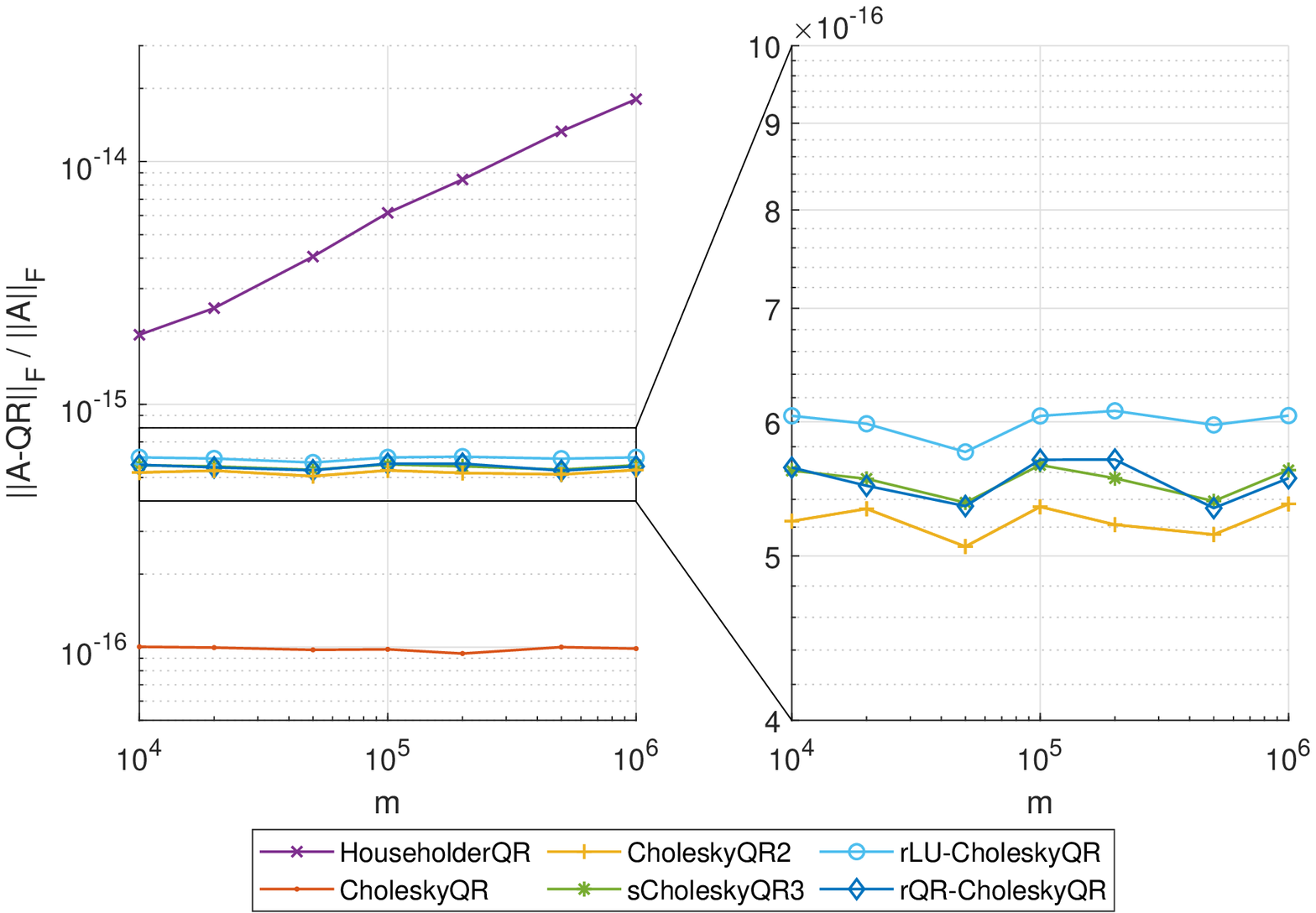}
	\caption{$\frac{\|A - QR\|_F}{\|A\|_F}$ for different algorithms when  $n = 100$, $\kappa = 10^{5}$ }
	\label{accuracy4}
\end{center}
\end{figure}

%As for the accuracy, the fig \ref{accuracy4} indicates except HouseholderQR other five methods would not become inaccurate with $m$ increasing. Although the accuracy of HouseholderQR has a trend of increasing, $\frac{\|A - QR\|_F}{\|A\|_F}$ is also smaller than $10^{-13}$ when decomposing matrix with $m = 10^6$, which is acceptable in most case.

%Last, we vary $n$ from 100 to 2000 with fix $m=10^5$ and $\kappa = 10^5$. As fig \ref{accuracy5} and fig \ref{accuracy6} show, the performance of CholeskyQR in orthogonality is still the worst though its decomposition accuracy $\frac{\|A - QR\|_F}{\|A\|_F}$ is slightly better than the others. The orthogonality and decomposition accuracy of rLU-CholeskyQR are both become worse and worse when $n$ becomes larger. The results of the other four algorithms don't change a lot when $n$ takes different value.

\begin{figure}[htbp]
\begin{center}
	\includegraphics[width=0.8\columnwidth]{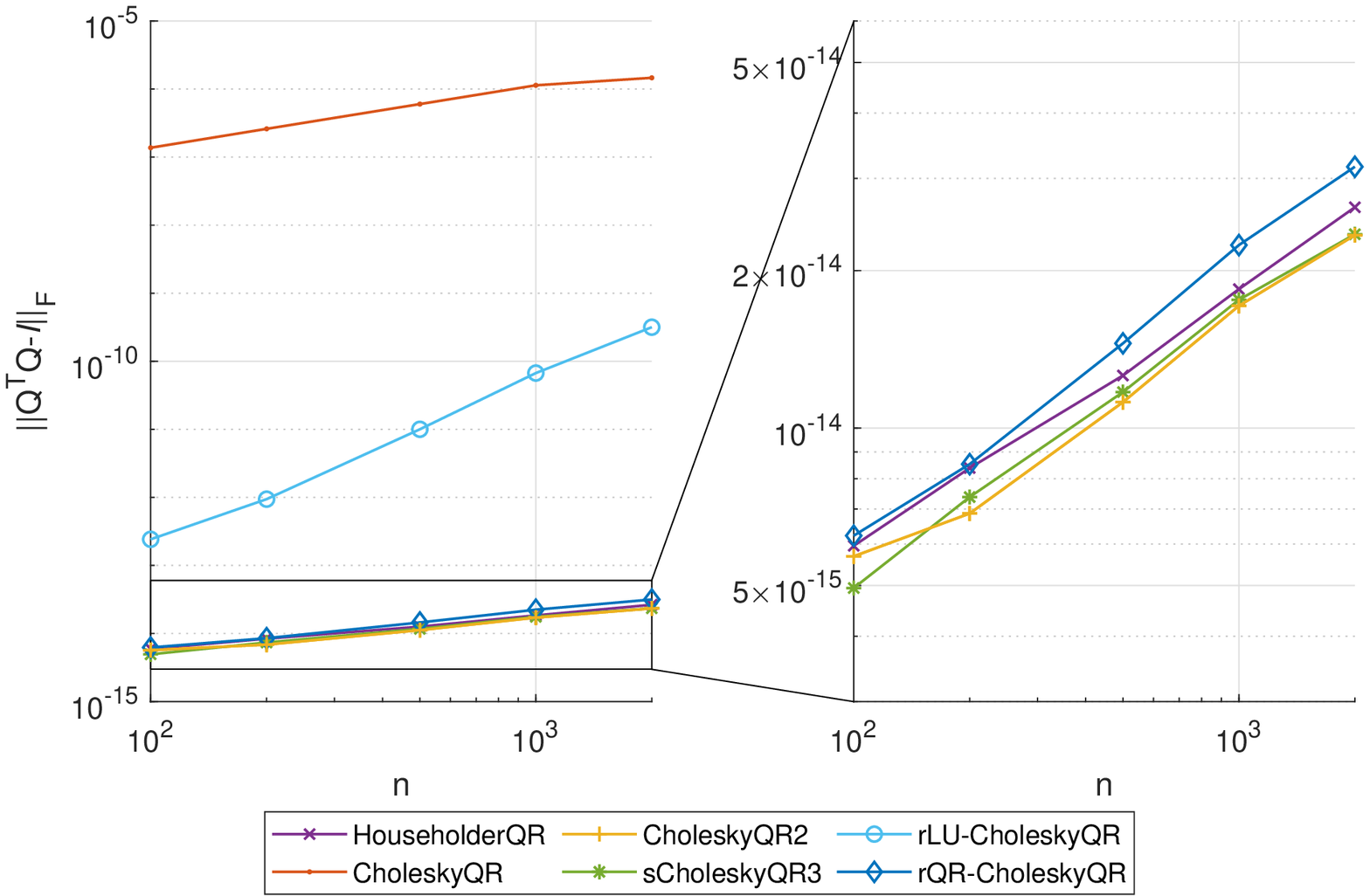}
	\caption{$\| Q^TQ - I\|_F$ for different algorithms when $m=10^5$, $\kappa = 10^{5}$.}
	\label{accuracy5}
\end{center}
\end{figure}

\begin{figure}[htbp]
	\label{accuracy6}
\begin{center}
	\includegraphics[width=0.8\columnwidth]{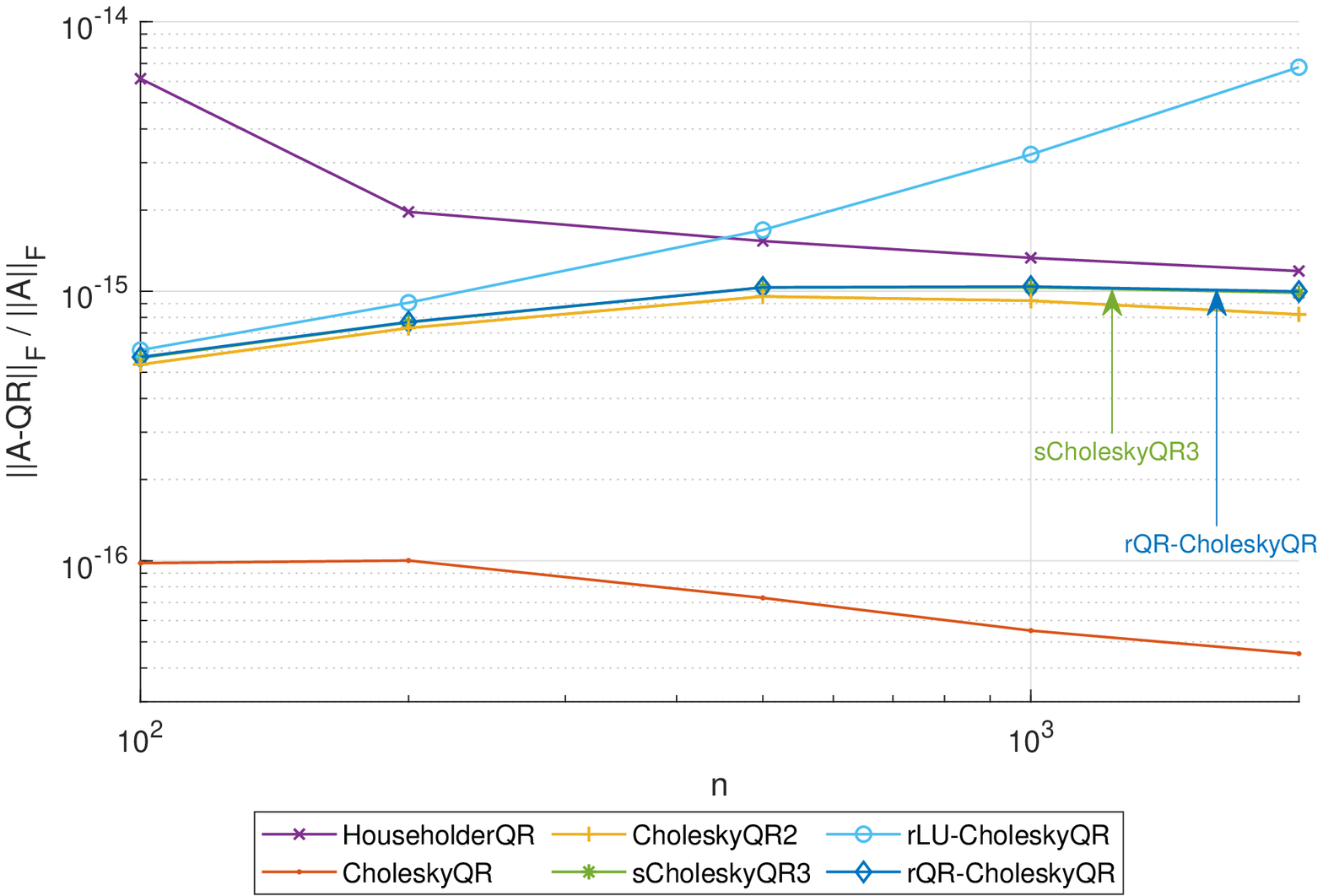}
	\caption{$\frac{\|A - QR\|_F}{\|A\|_F}$ for different algorithms when  $m=10^5$, $\kappa = 10^{5}$. }

\end{center}
\end{figure}

\subsection{Runtime test}
\label{sec:experiments:speed}
Now we turn to evaluate the runtime performance of our proposed methods compared to others. For single process, we take $n = 100, \kappa = 10^5$ and test several methods. As shown in \cref{speed1}, CholeskyQR outperforms other methods. Among all stable methods (\textit{i.e.}, exclude CholeskyQR), our proposed methods enjoy the best runtime performance, and the cost time of randomized LU/QR-CholeskyQR is about the same. It is worth noting that our methods are even faster than CholeskyQR2, which is proven limited in the previous subsection. All results in \cref{speed1} indicate that the computational cost is proportional to $m$, the size of the matrix, which corroborates the theoretical counting. 

\begin{figure}[htbp]
	\begin{center}
		\includegraphics[width=0.65\columnwidth]{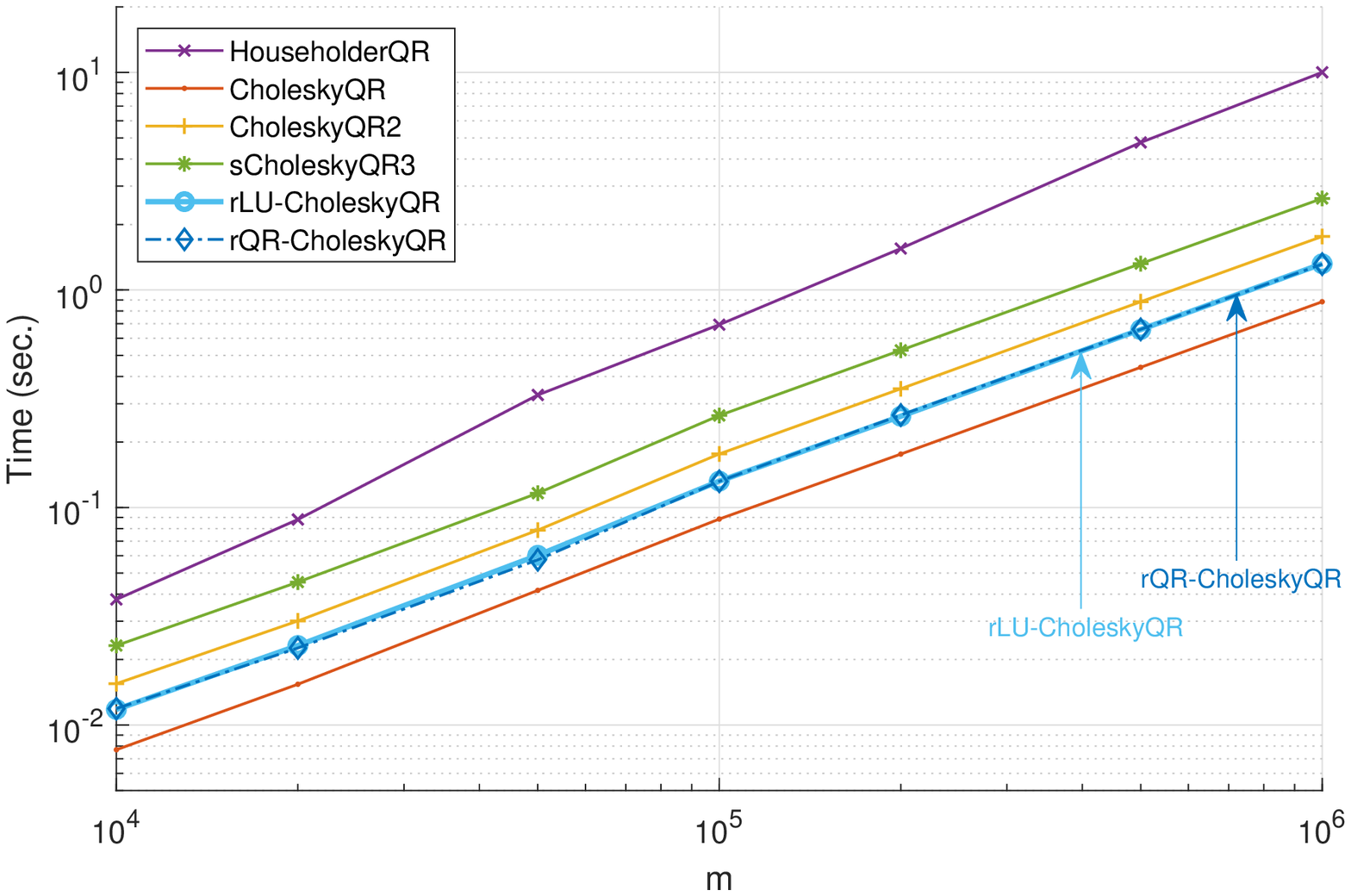}
		\caption{Runtime for different algorithms when $n = 100$, $\kappa = 10^5$, using 1 process.}
		\label{speed1}
	\end{center}
\end{figure}

To see whether similar results hold for multi-processes, we carry out the experiments in 32 processes. Here for convenience, we only choose methods that are easy to parallelize. We use the CholeskyQR2 algorithm instead of CholeskyQR for comparison, omit rLU-CholeskyQR since it has a similar performance as rQR-CholeskyQR. \cref{speed32_1} shows the result when taking $n = 100, \kappa  = 100,000$ while $m$ varies from $10^5$ to $10^7$. The runtime performance of those methods coincident with those in a single process. 
For completeness, we provide \cref{speed32_2}, illustrating the result when taking $n = 100, m = 100,000$ while $\kappa$ varies from $10^3$ to $10^{15}$. It indicates that CholeskyQR2 and sCholeskyQR3 take $24\%$ and $82\%$ more time than rQR-CholeskyQR, respectively. And CholeskyQR2 will collapse when the condition number is bigger than $10^8$.

\begin{figure}[htbp]
	\begin{center}
		\includegraphics[width=0.65\columnwidth]{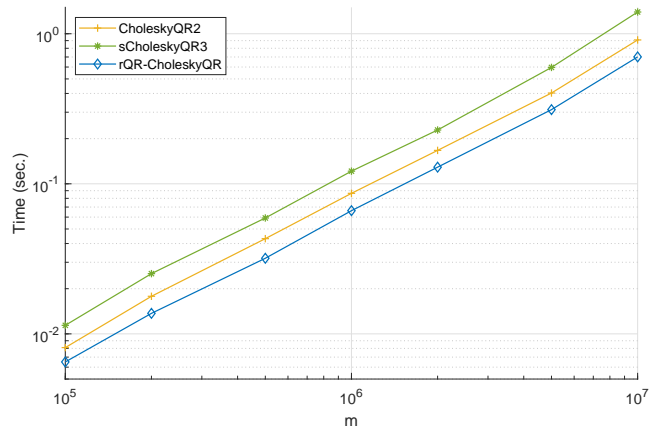}
		\caption{Runtime for different algorithms when $n = 100$, $\kappa = 10^5$, using 32 processes. }
		\label{speed32_1}
	\end{center}
\end{figure}

\begin{figure}[htbp]
	\begin{center}
		\includegraphics[width=0.65\columnwidth]{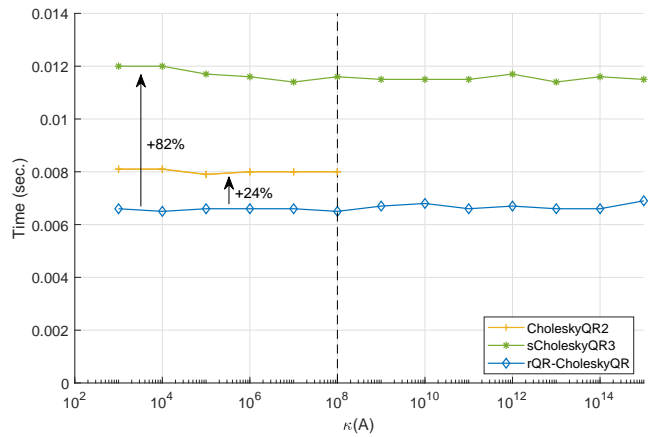}
		\caption{Time for different algorithms when $m=10^5$, $n = 100$, using 32 processes.}	
		\label{speed32_2}
	\end{center}
\end{figure}

\subsection{Scalability test}
\label{sec:experiments:scalability}
In this subsection, we examine the scalability of rQR-CholeskyQR, comparing to the CholeskyQR2 algorithm and the shifted CholeskyQR3 algorithm. We consider both strong scalability and weak scalability. The former focuses on how running time varies with the number of processors for a fixed total problem size. And the latter focuses on how the running time varies with the number of processors when fixing problem size tackled by each process.

As \cref{scalability1} shows, they all have good strong scalability. As the number of processes increases, the speedup ratio also increases, proportional to the number of processes. When using 512 processes, they speed up around 36 times than 8 processes, which indicates these methods can be used in the multi-processes configuration.

\begin{table}[htbp]
	\centering
	\caption{Computational time with different number of processes when $m = 10^7$, $n = 100$, $\kappa = 10^5$, speedup ratios are in brackets.}
	\begin{tabular}{c|ccc}	
		\hline
		processes & CholeskyQR2 & sCholeskyQR3 & rQR-CholeskyQR \\
		\hline
		8 & 2.3181 & 3.6855 & 1.8643 \\
		16 	& 1.4032 (1.65) 	& 2.2035 (1.67) 	& 1.1706 (1.59) \\
		32 	& 0.7886 (2.93) 	& 1.1413 (3.22) 	& 0.6152 (3.03) \\
		64 	& 0.4187 (5.53) 	& 0.5993 (6.14) 	& 0.3341 (5.58) \\
		128 & 0.2386 (9.71) 	& 0.3301 (11.16) 	& 0.1925 (9.68) \\
		256 & 0.1251 (18.52) 	& 0.1911 (19.28) 	& 0.1026 (18.17) \\
		512 & 0.0631 (36.74) 	& 0.0958 (38.47) 	& 0.052 (35.85) \\
		\hline
	\end{tabular}
	\label{scalability1}
\end{table}

The weak scalability can be easily derived from above results. The running time is proportional to the scale of matrix, which says if we double the row of matrix $m$, the running time will also double. But if we also double the number of processes $p$ at the same time, it will cost almost the same time as before or slightly more to compute. In short, fixing $n$, $\kappa$ and $m/p$, the runtime would only increase mildly when $m$ grows. As \cref{scalability2} shows, when $m$ varies from $10,000$ to $5,120,000$, the runtime of CholeskyQR2 and sCholeskyQR3 only become 2.26 and 2.25 times larger. Among these, rQR-CholeskyQR performs the best since it only consumes 2.21 times more time than before.

\begin{figure}[htbp]
	\begin{center}
		\includegraphics[width=0.8\columnwidth]{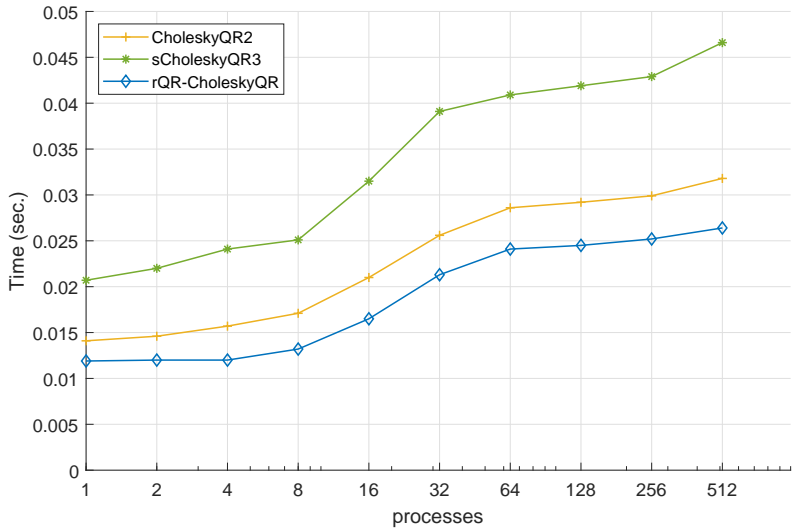}
		\caption{Time for different algorithms when $n = 100$, $\kappa = 10^5$ and fix $\frac{m}{p} = 10^5$}
		\label{scalability2}
	\end{center}
\end{figure}

\subsection{Dependence on numbers of sampling rows}
\label{sec:experiments:sampling}
Finally, we concentrate on the following problem: how many rows do we need to sample to ensure the condition numbers of the preconditioned matrices are small enough. Notice that in previous subsections, the sampling rate $\frac{l}{n}$ is selected as 2, a fixed number. In this section, we reveal the relation between the condition number (on expectation) and the sampling rate $\frac{l}{n} >1$. It seems essential in the sense that it has a strong impact on the final performance on preconditioning, but we will show our method is not so sensitive to this parameter, as soon as the sampling rate departs away from $1$.
%The selection of sampling rate will affect the condition number after precondition, then the accuracy of the result. 

As \cref{sampling} shows, a slight enlargement in sampling rate might lead to exponential decay in the condition number of $\hat Q$. When the sampling rate equals to 1, the condition number will possibly reach 3000 at most. However, if the sampling rate becomes a little larger, such as 1.2, the condition number of $\hat Q$ will drop down to 20, which is helpful in subsequent CholeskyQR. When the sampling rate continues to increase, the condition number after preconditioning will continue to drop, but not so significantly as before. For practical usage, we recommend choosing the sampling rate from 1.5 to 2.

\begin{figure}[htbp]
\begin{center}	
	\includegraphics[width=0.8\columnwidth]{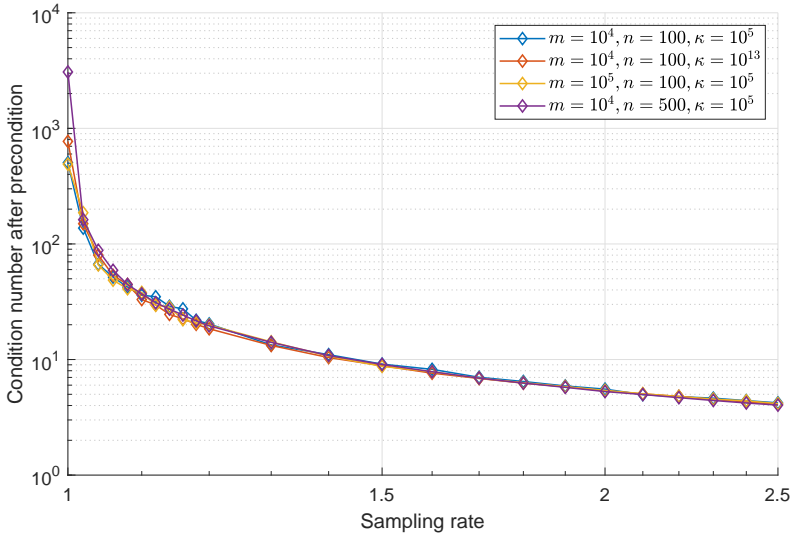}
	\caption{Condition number after preconditioning in four experiments}
	\label{sampling}
\end{center}
\end{figure}

\section{Applications}
\label{sec:Applications}

In this section, two examples are provided to explain the motivation of studying QR decomposition of tall-and-skinny matrices, which also serve as a natural application for our proposed methods. 

Nevertheless, this short section cannot encompass all cases using the proposed method. For example,  the computation of tall and skinny QR in LOBPCG takes up lots of time. Also, the block QR decomposition relies on tall-and-skinny matrices heavily. Hence it is worth mentioning that our algorithms can accelerate these numerical methods, and many of them are within the scope of future work.

\subsection{Randomized singular value decomposition}

Randomized SVD, proposed by \cite{granat2010novel}, aims to calculate several largest singular vector approximately for large scale matrices. It has been widely used and analyzed, see \cite{halko2011finding}. The basic version of RSVD is introduced as follows:
\begin{algorithm}[htbp]
\caption{Vanilla RSVD of $A \in \R^{m\times n}$}
\begin{algorithmic}[1]
\STATE Generate a random matrix $\Omega \in \R^{m\times k}$.
\STATE Set $Y = A^T\Omega \in \R^{n\times k}$.
\STATE Compute the SVD of $Y = \tilde U\Sigma V^T$.
\STATE Set $U = A\tilde U$.
\end{algorithmic}
\end{algorithm}

This algorithm might suffer from loosing accuracy. Instead of using vanilla version, a power method iteration is usually used to improve its accuracy, see \cref{alg:rsvd}. For further discussion about how to improve RSVD, see \cite{musco2015randomized}.
\begin{algorithm}[htbp] 
\caption{RSVD with Power Iteration}
\label{alg:rsvd}
\begin{algorithmic}[1]
\STATE Generate a random matrix $X \in \R^{n \times k}$.
\STATE $Y = AX.$
\STATE $Y = \text{orth}(Y)$.
\FOR{$k = 1$ to $K$}
\STATE $Y = A^TY.$
\STATE $Y = \text{orth}(Y)$.
\STATE $Y = AY$.
\STATE $Y = \text{orth}(Y)$.
\ENDFOR
\STATE Compute the SVD of $Y = \tilde U\Sigma V^T$.
\STATE Set $U = A\tilde U$.
\end{algorithmic}
\end{algorithm}

Often, the orthogonal step uses QR decomposition, for example, MATLAB built-in \texttt{qr} function. Therefore, we can use our proposed method instead. To test our proposed QR decomposition, we select several sparse matrices and report the time cost of each iteration. Here all the experiments are performed in MATLAB R2020b, on a personal laptop. The matrices we test are all selected from SuiteSparse Matrix Collection\footnote{\href{https://sparse.tamu.edu/}{https://sparse.tamu.edu/}.}, here we list the brief information of them, see \cref{table:info}. 

\begin{table}[htbp]
\centering
\caption{Information of test matrices.}
\label{table:info}
\begin{tabular}{|c|c|c|c|c|c|}
\hline
           & rows      & columns                                                                       & nnz                                                                            & density & nnz/row \\ \hline
Ge87H76    & 112,985   & 112,985                                                                       & 7,892,195                                                                      & 6.18e-4 & 69.85   \\ \hline
Si87H76    & 240,369   & 240,369                                                                       &  10,661,631 & 1.85e-4 & 44.36   \\ \hline
Hamrle3    & 1,447,360 & 1,447,360 &  5,514,242  & 2.63e-6 & 3.81    \\ \hline
G3\_circuit & 1,585,478 & 1,585,478                                                                     & 7,660,826                                                                      & 3.04e-6 & 4.83    \\ \hline
\end{tabular}
\end{table}
\Cref{tab:rsvd} shows the cost time in each iteration in \cref{alg:rsvd}. Notice that in each iteration, we simply replace MATLAB built-in QR (denoted as \textbf{QR}) and our proposed randomized QR-CholeskyQR (denoted as \textbf{rQR}), and the acceleration becomes more powerful, as the matrices becomes more sparse.
\begin{table}[htbp]
\centering
\caption{The comparison between using QR and rQRCholeskyQR.}
\label{tab:rsvd}
\begin{tabular}{|c|c|c|c|c|c|c|c|c|c|}
\hline
\multirow{2}{*}{time(s)} & \multicolumn{3}{c|}{$k = 20$} & \multicolumn{3}{c|}{$k = 50$} & \multicolumn{3}{c|}{$k = 100$} \\ \cline{2-10} 
                      & QR  & rQR & ratio & QR  & rQR & ratio & QR  & rQR  & ratio \\ \hline
Ge87H76               & 0.308     & 0.213   & 1.45  &  0.741     &  0.557     & 1.33  & 2.087     & 1.509      & 1.38  \\ \hline
Si87H76               & 0.540     & 0.321    & 1.68  & 1.466     & 0.912    & 1.61  & 3.341     & 2.021     & 1.65 \\ \hline
Hamrle3               & 2.528    & 0.762   & 3.27 & 6.790   & 2.653     & 2.56 & 12.392   & 4.559   & 2.72  \\ \hline
G3\_circuit            & 2.904    & 0.903   & 3.21 & 5.964   & 2.406    & 2.48 & 15.128   & 5.135     & 2.95  \\ \hline
\end{tabular}
\end{table}

\subsection{Least square}
In this section, we extend our interest in the $R$ factor of randomized QR. Consider the Least Square problem of $Ax = b$, which is usually solved by the normal equation 
$x = (A^TA)^{-1}A^Tb$, provided $A$ is of the full rank. However, the solver of the normal equation also suffers from ill-conditioned $A^TA$. Hence a suitable preconditioner is required.
As we remarked before, the $R$ factor of the submatrix, denoted as $R_1$, can be regarded as a suitable preconditioner since $\cond(AR_1^{-1})$ is small with high probability. The algorithm using the randomized $R$ factor can be summarized as follows.

\begin{algorithm}[htbp]
\caption{LS solver preconditioned by randomized $R$ factor}
\begin{algorithmic}[1]
\STATE Compute $B = AR_1^{-1}$.
\STATE Solve LS problem $By = b$ by normal equation: $y = (B^TB)^{-1}B^Tb$.
\STATE Solve $x$ from $y$: $x = R_1^{-1}y$.
\end{algorithmic}
\end{algorithm}

% \section{Discussions}
% \label{sec:discussions}

\appendix

\section{Estimations on rQR-CholeskyQR under inexact procedure}
\label{app:round-off-error}

In this section, we focus on the estimation in \cref{alg:rqr-CholeskyQR} when the sub-routine is not exactly performed. The inexactness often comes from round-off error, which will be mainly discussed in this section. The main spirit of this section is that, under a stable choice of sub-routine, the result is also stable, in the sense that the condition number of computed $X$ (denoted by $\hat{X}$ for clarity) enjoys a mild increase than $X$. We first assume that our subroutine (QR and triangular solver) are forward stable under  $\ell^2$ norm.

% \begin{definition}
% A $R$ factor $\hat{R}$ of $V \in \R^{l \times n}$ is $\alpha$-stable if there exist an orthogonal matrix $\underline{Q}$ such that $V + \Delta V = \underline{Q} \hat{R}$ and the following estimates holds:
% \begin{equation}
% \|\Delta V\|_{2} \le \alpha \|V\|_2.
% \end{equation}
% \end{definition}
\begin{definition}
The QR decomposition is $\alpha$-stable for inexact QR
\[ A \approx (Q + \Delta Q) (R + \Delta R) ,\]
we have  
\[ \|\Delta R R^{-1}\|_2 \le \alpha \cond(R). \]
\end{definition}

\begin{definition}
$Y \approx AR^{-1}$ is $\beta$-stable with respect to triangular system $(A,R)$ if the following estimation holds 
\begin{equation}
\|\Delta Y\|_2 / \|Y\|_2 \le \beta \cond(A),
\end{equation}
where $Y + \Delta Y = AR^{-1}$.
\end{definition}

We now give a framework on inexact version of \cref{alg:rqr-CholeskyQR}.

\begin{algorithm}[htbp]
	\caption{$[Q,R]$ = \aqrcholqr($A$) (Inexact Version)}
	\label{alg:rqr-CholeskyQR-inexact}
	\begin{algorithmic}[1]
	\STATE Choose $A_1$ as a submatrix of $A$, with size $l\times n$.
	\STATE $[\sim,\hat{R}_1] \approx \qr(A_1)$
	\STATE $\hat{X} \approx A\hat{R}_1^{-1}$. 
	\STATE $[Q,R_2]$ = \cholqr($\hat{X}$).
	\STATE $R = R_2\hat{R}_1$.
	\end{algorithmic}
	\end{algorithm}

\begin{proposition}
	Suppose that for a submatrix $A_1$, $\hat{R}_1$ is an $\alpha$-stable $R$ factor of $A$ via QR decomposition. $\hat{X}$ is $\beta$-stable with respect to $(A,\hat{R}_1)$. We assume that $\beta\cond(A) < \frac{1}{2}, \alpha\cond(A_1) < \frac{1}{2}$, then we have 
	\begin{equation}
	\|X - \hat{X}\|_2 \le 2(\beta \cond(A) + \alpha\cond(A_1)) \|X\|_2.
	\end{equation}
	Moreover, denote by $\gamma = \|X - \hat{X}\|_2/ \|X\|_2 \le 2(\beta \cond(A) + \alpha\cond(A_1))$, we assume $\gamma \cond(X) <1$, then 
	\begin{equation}
		\cond(\hat{X}) \le \frac{1+\gamma}{1-\gamma} \cond(X).
	\end{equation}
	\end{proposition}

\begin{proof}
We denote by $Y:= A\hat{R}_1^{-1}$ the exact result of triangular system. Then by the definition of $\beta$-stability, we have
\begin{equation}
\|Y - \hat{X}\|_2 \le \beta \cond(A) \|Y\|_2
\end{equation}
Next, we estimate the term $\|X - Y\|_2$, it follows that 
\begin{equation}
	\begin{split}
\|X - Y\|_2 = \|AR^{-1} - A \hat{R}^{-1} \|_2 \le \|AR^{-1}\|_2 \|I - \hat{R}_1 R_1^{-1}\| \le \alpha \cond(R_1)\|Y\|_2 
	\end{split}
\end{equation}

Since $\cond(R_1) = \cond(A_1)$, then by triangle inequality we have 
\begin{equation}
\|Y\|_2 \le \frac{1}{1 - \alpha \cond(A_1)}\|X\|_2
\end{equation}
Therefore 
\begin{equation}
	\begin{split}
\|X - \hat{X}\|_2 \le~&  \|X - Y\|_2 + \|Y - \hat{X}\|_2  \\  \le~& \beta \cond(A) \|Y\|_2 + \alpha \cond(A_1) \|Y\|_2  \\ \le~&  
\frac{\beta \cond(A) + \alpha \cond(A_1)}{1 - \alpha \cond(A_1)} \|X\|_2  \\
\le~& 2(\beta\cond(A) + \alpha\cond(A_1))\|X\|_2.
	\end{split}
\end{equation}

Hence, by Weyl's theorem \cite[Chapter 8.6]{golub2013matrix}
\begin{equation}
	\cond(\hat{X}) \le \frac{\sigma_1(X) + \gamma \|X\|_2}{\sigma_n(X) - \gamma\|X\|_2} \le \cond(X) \frac{1+\gamma}{1 - \gamma}.
	\end{equation}
\end{proof}

We close this section by discussing the aforementioned stability. Under mild assumptions,
the backward stability for triangular systems (see \cite[Theorem 8.5]{higham2002accuracy}) can tell us that 
$\|\Delta X\|_F \le cm\u \cond(A)\|X\|_F$, and thus it is $\beta$-stable where $\beta \le cm\sqrt{n}\u${ since $\|A\|_{2} \le \|A\|_F \le \sqrt{n}\|A\|_2$.}

The $\alpha$-stability is more involved. We first recall the backward stability of Householder and Givens QR, (see \cite[Chapter 18]{higham2002accuracy}). That means, there exists $\Delta A$, such that $\|\Delta A\|_F \le \varepsilon \|A\|_F$, such that $\hat{R}$ is the $R$ factor of $A + \Delta A$. Here
\begin{itemize}
\item For Householder QR, $\varepsilon \le cln\u$, where $c$ is a small constant.
\item For Givens QR, $\varepsilon \le c(l+n)\u$, where $c$ is a small constant.
\end{itemize}

By expanding the definition of backward stability, we have 
\begin{equation}
QR + \Delta A = \underline{Q} \hat{R},
\end{equation}
or equivalently,
\begin{equation}
	I + Q^{-1}\Delta AR^{-1} = Q^{-1}\underline{Q} \hat{R}R^{-1} = (Q^{-1}\underline{Q})(I+\Delta RR^{-1}).
\end{equation}

By perturbation result in \cite[Theorem 3.1]{chang1997perturbation}, we then have
\begin{equation}
\frac{\|\Delta R R^{-1}\|_F}{\|R\|_2}\le c \|Q^{-1}\Delta AR^{-1}\|_F \le \|\Delta A\|_F\|R^{-1}\|_2.
\end{equation}

Combining backward stability we have
\[  \|\Delta R R^{-1}\|_2 \le \varepsilon \|A\|_F \cond(R), \]
and the common result on $\varepsilon$ has been discussed before.

% \section*{Acknowledgments}
% The authors would like to thank...

\bibliographystyle{siamplain}
\bibliography{ref}
\end{document}